\documentclass{aims}

\usepackage{latexsym, amsfonts,mathrsfs, amsmath, amssymb, amscd, epsfig}
\usepackage{mathrsfs}

\usepackage{amsthm}
\usepackage[all,cmtip,line]{xy}
\usepackage{array}
\usepackage{psfrag}
\usepackage{overpic}
\usepackage{bm}
\usepackage{setspace}

\usepackage{bm}

 \usepackage{paralist}
  \usepackage{graphics} 
  \usepackage{epsfig} 
\usepackage{graphicx}  \usepackage{epstopdf}
 \usepackage[colorlinks=true]{hyperref}
\hypersetup{urlcolor=blue, citecolor=red}

\def\currentvolume{X}
 \def\currentissue{X}
  \def\currentyear{200X}
   \def\currentmonth{XX}
    \def\ppages{X--XX}
     \def\DOI{10.3934/xx.xx.xx.xx}

\numberwithin{equation}{section}
\newtheorem{thm}{Theorem}[section]
\newtheorem{prop}[thm]{Proposition}
\newtheorem{lem}[thm]{Lemma}

\newtheorem{remark}[thm]{Remark}
\newtheorem{definition}[thm]{Definition}

\newcommand {\qd} {\quad}

\newcommand \pt {\partial}

\newcommand \grad{\nabla}
\newcommand \gam{\gamma}
\newcommand \R{\mathbb{R}}
\newcommand \Om{\Omega}
\newcommand \der{\partial}

\newcommand \mcl{\mathcal}
\newcommand \Gam{\Gamma}
\newcommand \alp{\alpha}
\newcommand \tx{\text}
\newcommand \til{\tilde}
\newcommand \p {\prime}

\newcommand \ol{\overline}

\newcommand \bu {{\bf u}}

\newcommand \ul {\underline }
\newcommand \eps {\varepsilon}
\newcommand \nn {\mathcal A}

\newcommand \shock{\Gamma_s}
\newcommand \bmnu{{\bm\nu}_s}
\newcommand \bmtau{{\bm\tau}_s}

\def\mE{\mathcal{E}}

\title[Transonic shocks of E-P system in convergent nozzles] 
      {Radial transonic shock solutions of Euler-Poisson system in convergent nozzles}
\author[Myoungjean Bae and Yong Park]{}



 \email{mjbae@postech.ac.kr}
 \email{pipablue@postech.ac.kr}

\thanks{The first author is supported in part by Samsung Science and Technology Foundation
under Project Number SSTF-BA1502-02. The second author supported in part under the framework of international cooperation program managed by National Research Foundation of Korea(NRF-2015K2A2A2002145).
}

\thanks{$^*$ Corresponding author: xxxx}

\begin{document}

\maketitle

\centerline{\scshape Myoungjean Bae$^*$}
\medskip
{\footnotesize
 \centerline{Department of Mathematics,
         POSTECH, Pohang, Gyungbuk, 37673, Republic of Korea,}
          \centerline{
         Korea Institute for Advanced Study
85 Hoegiro, Dongdaemun-gu,
Seoul 02455,
Republic of Korea}
} 

\medskip

\centerline{\scshape Yong Park}
\medskip
{\footnotesize
 \centerline{Department of Mathematics,
         POSTECH, Pohang, Gyungbuk, 37673, Republic of Korea,}
}

\bigskip

 \centerline{(Communicated by )}

\begin{abstract}
Given constant data of density $\rho_0$, velocity $-u_0{\bf e}_r$, pressure $p_0$ and electric force $-E_0{\bf e}_r$ for supersonic flow at the entrance, and constant pressure $p_{\rm ex}$ for subsonic flow at the exit, we prove that Euler-Poisson system admits a unique transonic shock solution in a two dimensional convergent nozzle, provided that $u_0>0$, $E_0>0$, and that $E_0$ is sufficiently large depending on $(\rho_0, u_0, p_0)$ and the length of the nozzle.

\end{abstract}
\section{Introduction}
The steady Euler-Poisson system
\begin{equation}
\left\{
\begin{array}
[c]{l}%
\text{div}(\rho\mathbf{u})=0,\\
\text{div}(\rho\mathbf{u}\otimes\mathbf{u})+\nabla p=\rho\nabla\Phi,\\
\text{div}(\rho\mE \bu +p\bu)=\rho\mathbf{u}\cdot\nabla\Phi,\\
\Delta\Phi=\rho-\tilde{b},
\end{array}
\right.  \label{ep}%
\end{equation}
describes the motion of electrons governed by self-generated electric field in macroscopic scale. In \eqref{ep}, $\bu,
\rho$, $p$, and $\mE$ represent the macroscopic particle velocity, density,
pressure, and the total energy density, respectively. And, $\Phi$ represents the electric potential generated by the
Coulomb force of particles. The function $\tilde{b}>0$ is fixed, and represents the density of positively charged background ions. In this work, we consider ideal polytropic gas for which the pressure $p$ and the energy density $\mE$  are given by
\begin{equation}
\label{2d-1-b1}
p(\rho, S)=e^S {\rho^{\gam}}\quad \text{and}\quad
\mE({\bf u}, \rho, S)= \frac 12|{\bf u}|^2+\frac{e^S \rho^{\gam-1}}{\gam-1},
\end{equation}
respectively. Here, the constant $\gam>1$ is called the {adiabatic exponent}, and $S>0$ represents the entropy.
\smallskip

The goal of this work is to construct a family of radial transonic shock solutions to \eqref{ep} in a two dimensional convergent nozzle, and to study various analytical features especially including the monotonicity property of the pressure at the exit with respect to shock location, provided that the magnitude of electric field at the entrance is fixed sufficiently large.
\smallskip

\begin{definition}[A shock solution of E-P system]
\label{definition-shock-sol}
Let $\Om\subset \R^2$ be an open and connected set, and suppose that a $C^1$ curve $\shock$ divides $\Om$ into two open and connected sub-domains $\Om^+$ and $\Om^-$ so that $\Om^+\cup \shock \cup \Om^-=\Om$. Let ${\bmnu}$ be the unit normal vector field on $\shock$ oriented into $\Om^+$, and let $\bmtau$ be a tangent vector field on $\shock$.
We call ${\bf U}:=(\rho,\bu,p, \Phi)\in [L^\infty_{\rm{loc}}(\Om)\cap C^0(\ol{\Om^\pm})\cap C^1(\Om^\pm)]^4\times [W^{1,\infty}_{\rm{loc}}(\Om) \cap C^1(\ol{\Om^\pm})\cap C^2(\Om^\pm)]$ with $\rho>0$ in $\ol{\Om}$ a {\emph{shock solution}} to \eqref{ep} in $\Om$ with a shock $\shock$ if ${\bf U}$ satisfies \eqref{ep} pointwisely in $\Om^{\pm}$, and satisfies the following {\emph{extended Rankine-Hugoniot jump conditions}}
\begin{align}\label{rhb}
&[\rho \bu\cdot\bmnu]_{\shock}=[\bu\cdot\bmtau]_{\shock}=
[\rho(\bu\cdot \bmnu)^2+p]_{\shock}=[B]_{\shock}=0,\\
\label{rh-Phi}
&[\Phi]_{\shock}=[\grad\Phi\cdot\bmnu]_{\shock}=0,
\end{align}
for the Bernoulli function $B$ given by
$$B=\frac 1 2|\bu|^2+\frac{\gam p}{(\gam-1)\rho}=
\frac 1 2|\bu|^2+\frac{\gam e^{S}\rho^{\gam-1}}{\gam-1}.
$$
In \eqref{rhb}--\eqref{rh-Phi}, $[F]_{\shock}$ is defined by
$[F({\rm x})]_{\shock}:=F({\rm x})|_{\ol{\Om^-}}-F({\rm x})|_{\ol{\Om^+}}$ for ${\rm x}\in \shock$.
\end{definition}
One can easily extend Definition \ref{definition-shock-sol} to the case of $\Om\subset \R^n$ with $n\ge 3$ through replacing a $C^1$ curve $\shock$ and a tangent vector field $\bmtau$ on $\shock$ by a $C^1$ (hyper)surface $\shock$ and tangent vector fields $\{\bmtau^{(j)}\}_{j=1}^{n-1}$ on $\shock$ with $\{\bmtau^{(j)}\}_{j=1}^{n-1}$ being linearly independent at each point on $\shock$, respectively.

\begin{definition}[Admissibility of a shock solution]
\label{definition-admissibility}
Let $(\rho_{\pm}, {\bu}_{\pm}, p_{\pm}, \Phi_{\pm})$ denote $(\rho, {\bu}, p, \Phi)$ restricted on $\Om^{\pm}$, respectively.
A shock solution ${\bf U}$ with a shock $\shock$ is said {\emph{physically admissible}} if
\begin{equation*}
\begin{split}
  &0<{\bu_+}\cdot\bmnu<{\bu_-}\cdot \bmnu\quad\tx{on}\,\,\shock,\\
  \tx{or equivalently}\quad &{\bu_+}\cdot \bmnu>0, {\bu_-}\cdot\bmnu>0\,\,\tx{and}\,\,S_+>S_-\quad\tx{on}\,\,\shock
    \end{split}
\end{equation*}
for $\ln S=\frac{p}{\rho^{\gam}}$.
\end{definition}
Depending on the {\emph{Mach number}} $M=\frac{|{\bu}|}{c(\rho, S)}$ for $c(\rho,S)=\sqrt{\gam e^{S}\rho^{\gam-1}}$, analytic and physical features of \eqref{ep} vary.
If $M>1$, then the corresponding state ${\bf U}=(\rho, {\bu}, p, \Phi)$ is called {\emph{supersonic}}. If $M<1$, on the other hand, then the corresponding state is called ${\emph{subsonic}}$. Here, $c(\rho, S)$ is called the {\emph{local sound speed}}.
\begin{definition}[Transonic shock solution]
\label{definition-transonic-shock}
A shock $\shock$ is called a {\emph{transonic shock}} if $M$ jumps from the state of $M>1$ to the state of $M<1$ across $\shock$.
\end{definition}
\begin{figure}[htp]
\centering
\begin{psfrags}
\psfrag{r0}[cc][][0.7][0]{$\phantom{aaaaaaaaaaaaaa}\Gam_{\rm ent}: M_0>1, E_0>0$}
\psfrag{rs}[cc][][0.7][0]{$r_s$}
\psfrag{r1}[cc][][0.7][0]{$\phantom{aaaaaaaa}\Gam_{ex}:p=p_{\rm ex}$}
\psfrag{eo}[cc][][0.7][0]{$\phantom{aaa}-E_{-}{\hat{\bf r}}$}
\psfrag{ei}[cc][][0.7][0]{$\phantom{aaa}-E_{+}{\hat{\bf r}}$}
\psfrag{ui}[cc][][0.7][0]{$-u_{+}{\hat{\bf r}}\phantom{a}$}
\psfrag{uo}[cc][][0.7][0]{$-u_{-}{\hat{\bf r}}$}
\psfrag{sub}[cc][][0.7][0]{$M_+<1$}
\psfrag{super}[cc][][0.7][0]{$M_->1$}
\psfrag{sh}[cc][][0.7][0]{\phantom{a}${\bf{\shock:\tx{shock}}}$}
\includegraphics[scale=0.6]{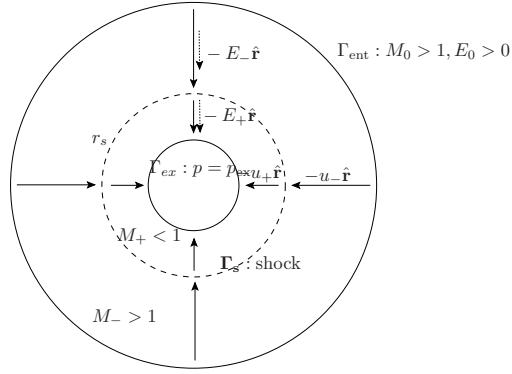}
\caption{Transonic shock of E-P system in a convergent domain($E_{\pm}>0, u_{\pm}>0$)}\label{figure1}
\end{psfrags}
\end{figure}
The goal of this work is to show that, given constant data of $(\rho_0, -u_0{\bf e}_r, p_0)$ and electric force $\nabla\Phi=-E_0{\bf e}_r$ for supersonic flow at the entrance $\Gam_{\rm{ent}}$, and constant pressure $p_{\rm ex}$ for subsonic flow at the exit $\Gam_{\rm{ex}}$, the system \eqref{ep} admits a unique radial transonic shock solution in a two dimensional annulus of finite radius, provided that $\rho_0$, $u_0$, $p_0$, $E_0$ and $p_{\rm{ex}}$ are positive, and that $E_0$ is sufficiently large depending on $(\rho_0, u_0, p_0)$ and the length of the nozzle. See Figure \ref{figure1}.
\medskip

\begin{figure}[htp]
\centering
\begin{psfrags}
\psfrag{px}[cc][][0.7][0]{$\phantom{aaa}\Gam_{ex}:p=p_{\rm ex}$}
\psfrag{rs}[cc][][0.7][0]{$r_s$}
\psfrag{r1}[cc][][0.7][0]{$\phantom{a}r_{\rm{in}}$}
\psfrag{eo}[cc][][0.7][0]{$\phantom{aaa}-E_{-}{\hat{\bf r}}$}
\psfrag{ei}[cc][][0.7][0]{$\phantom{aaa}-E_{+}{\hat{\bf r}}$}
\psfrag{ui}[cc][][0.7][0]{${\bf u}=u_{-}{\hat{\bf r}}\phantom{a}$}
\psfrag{uo}[cc][][0.7][0]{${\bf u}=u_{+}{\hat{\bf r}}$}
\psfrag{sub}[cc][][0.7][0]{$M_+<1$}
\psfrag{sup}[cc][][0.7][0]{$M_->1$}
\psfrag{m0}[cc][][0.7][0]{$M_0>1\phantom{a}$}
\psfrag{sh}[cc][][0.7][0]{\phantom{a}${\shock}$}
\psfrag{ul}[cc][][0.7][0]{${\bf u}=u_-{\bf e}_x$}
\psfrag{ur}[cc][][0.7][0]{${\bf u}=u_+{\bf e}_x$}
\includegraphics[scale=1.0]{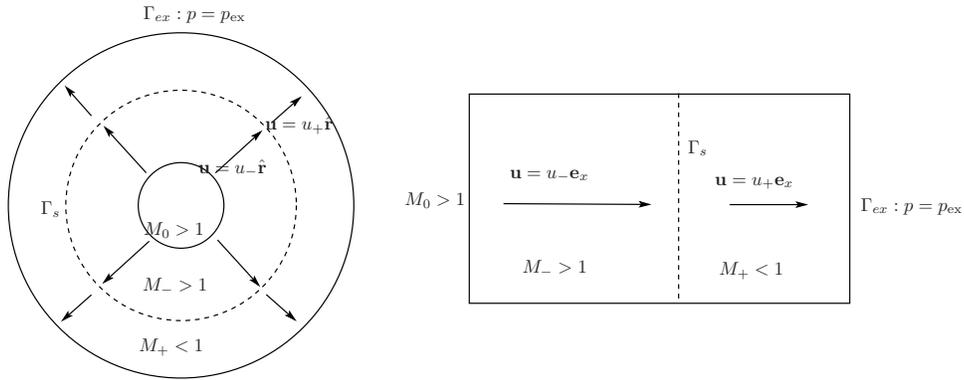}
\caption{Transonic shocks of Euler system(left: well-posed, right: ill-posed)}\label{figure2}
\end{psfrags}
\end{figure}
In \cite{HYuan}, by nonlinear ODE analysis of Euler system in two different domains (a flat nozzle, and a divergent nozzle) in $\R^2$, it is shown that
the geometry of a domain plays a key role to decide analytic behavior of transonic shock solutions. In a divergent nozzle, for fixed entrance data of supersonic flow and for fixed constant pressure at the exit, it is proved that steady Euler system
\begin{equation}
\label{euler-system}
\begin{cases}
\text{div}(\rho\mathbf{u})=0,\\
\text{div}(\rho\mathbf{u}\otimes\mathbf{u})+\nabla p=0,\\
\text{div}(\rho\mE \bu +p\bu)=0
\end{cases}
\end{equation}
admits a unique radial transonic shock solution, provided that the exit pressure is in a certain range for subsonic flow. In a flat nozzle, on the other hand, a transonic shock problem of Euler system with constant boundary data has either no solution or infinitely many solutions. See Figure \ref{figure2}.
Overall, the geometry of a domain determines whether a transonic shock problem of Euler system with fixed exit pressure is well-posed or not.

\begin{figure}[htp]
\centering
\begin{psfrags}
\psfrag{px}[cc][][0.7][0]{$\phantom{aaaaa}\Gam_{ex}:p=p_{\rm ex}$}
\psfrag{rs}[cc][][0.7][0]{$r_s$}
\psfrag{r1}[cc][][0.7][0]{$\phantom{a}r_{\rm{in}}$}
\psfrag{el}[cc][][0.7][0]{$\nabla\Phi=E_{-}{{\bf e}_x}$}
\psfrag{er}[cc][][0.7][0]{$\nabla\Phi=E_{+}{\bf e}_x$}
\psfrag{ui}[cc][][0.7][0]{${\bf u}=u_{-}{\hat{\bf r}}\phantom{a}$}
\psfrag{uo}[cc][][0.7][0]{${\bf u}=u_{+}{\hat{\bf r}}$}
\psfrag{sub}[cc][][0.7][0]{$M_+<1$}
\psfrag{sup}[cc][][0.7][0]{$M_->1$}
\psfrag{m0}[cc][][0.7][0]{$M_0>1, E_0>0\phantom{aaaaa}$}
\psfrag{sh}[cc][][0.7][0]{\phantom{a}${\shock}$}
\psfrag{ul}[cc][][0.7][0]{${\bf u}=u_-{\bf e}_x$}
\psfrag{ur}[cc][][0.7][0]{${\bf u}=u_+{\bf e}_x$}
\includegraphics[scale=1.0]{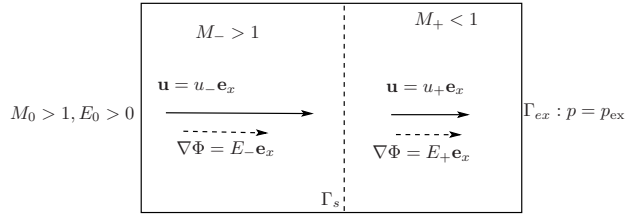}
\caption{Transonic shock of E-P system in a flat nozzle (conditionally well-posed) }\label{figure3}
\end{psfrags}
\end{figure}

Interestingly, \cite{LRXX} and \cite{LX} reveal that if the background charge density $b$ is less than the sonic density, then one dimensional transonic shock solutions of steady isentropic Euler-Poisson system defined in a flat nozzle have a similar feature to radial transonic shock solutions of Euler system in a divergent nozzle. More specifically, for fixed constant entrance data for supersonic flow and constant pressure for subsonic flow at the exit in a flat nozzle, steady isentropic Euler-Poisson system admits a unique one dimensional transonic shock solution, provided that the electric field at the entrance acts in the same direction as the entrance flow velocity. Also, one can directly check that the same is true for the system \eqref{ep}. See Figure \ref{figure3}.
This conditional well-posedness of a transonic shock problem of Euler-Poisson system with a fixed exit subsonic pressure shows that the self-generated electric field $\nabla\Phi$ with $\nabla\Phi\cdot{\bf u}>0$ in the system \eqref{ep} has the same effect of the geometry of a divergent nozzle in Figure \ref{figure2}. This naturally raises a question on {\emph{the well-posedness of a transonic shock problem of compressible Euler-Poisson system in a convergent nozzle with a fixed exit subsonic pressure  under a strong effect of electric field}. And, it is our goal to establish the well-posedness of radial transonic shock problem of \eqref{ep} in a convergent domain under a strong effect of self-generated electric field. If this well-posedness of a transonic shock problem in a convergent nozzle is achieved, then one can expect to  establish the dynamical stability of a transonic shock of the system \eqref{ep} in a convergent nozzle. In \cite{TLiu}, it has been shown that a one-dimensional transonic shock of Euler system in a convergent nozzle is dynamically unstable. Therefore, the results in this paper show that self-generated force in compressible flow can overcome the geometry of domain to stabilize a certain physical feature of the flow.

\section{Radial transonic shock solutions of \eqref{ep} in a convergent domain}

Let $(r,\theta)$ be polar coordinates in $\R^2$. For fixed constants $r_0>r_1>0$, define an annular domain
\begin{equation}
\label{def-annulus}
\nn:=\{{\rm x}\in \R^2:r_1<|{\rm x}|<r_0\}.
\end{equation}
Set
\begin{equation*}
\Gam_0:=\der\nn\cap\{|{\rm x}|=r_0\},\quad \Gam_1:=\der\nn\cap\{|{\rm x}|=r_1\}.
\end{equation*}
In \eqref{ep}, assume that $\tilde b$ is in $C^1(\ol{\nn})$ and $\tilde b=\tilde b(r)$ with
\begin{equation*}
  \|\tilde b\|_{C^1(\ol{\nn})}\le b_0
\end{equation*}
for some constant $b_0>0$.
Fix $\gam>1$. For positive constants $(\rho_0, u_0, p_0, E_0)$, we prescribe boundary conditions as follows:
\begin{equation}\label{BCs-enterance}
(\rho, {\bf u}, p, \nabla \Phi)=(\rho_0, -u_0{\bf e}_r, p_0, -E_0 {\bf e}_r)\quad\tx{on}\,\,\Gam_0
\end{equation}
for ${\bf e}_r=\frac{\rm x}{|{\rm x}|}$. Since \eqref{ep} and \eqref{BCs-enterance} are invariant under a coordinate rotation, we set as
\begin{equation*}
  (\rho, {\bf u}, p, \nabla\Phi)({\rm x})=(\tilde{\rho}(r), \tilde{u}(r) {\bf e}_r, \tilde{p}(r), \tilde{E}(r){\bf e}_r ),
\end{equation*}
so that \eqref{ep} and \eqref{BCs-enterance} are rewritten as
\begin{equation}\label{rdeq}
\begin{split}
&\begin{cases}
\frac{d}{dr}(r\tilde{\rho}\tilde{u})=0\\
\frac{d}{dr}(r\tilde{\rho}\tilde{u}^2)+r\frac{d\tilde{p}}{dr}=r\tilde{\rho}\tilde{E}\\
\frac{d}{dr}(r\tilde{\rho}\tilde{u}{B})=r\tilde{\rho}\tilde{u}\tilde{E}\\
\frac{d}{dr}(r\tilde{E})=r(\tilde{\rho}-\tilde{b}(r))
\end{cases}\quad\tx{for}\,\,r_1<r<r_0,\\
&(\tilde{\rho}, \tilde{u}, \tilde{p}, \tilde{E})(r_0)=(\rho_0, -u_0, p_0, -E_0).
\end{split}
\end{equation}
As we seek solutions flowing in the direction of $-{\bf e}_r$ in $\nn$, it is convenient to introduce new variables
\begin{equation}\label{cv}
(t,\hat t)=(r_0-r,r),
\end{equation}
\tx{and to set}
\begin{equation*}
(\rho, u, p, E, b)(t):=(\tilde \rho,-\tilde u,\tilde p, -\tilde E, \tilde b)(r).
\end{equation*}
Then \eqref{rdeq} is equivalent to the following initial value problem for $(\rho, u, p, E)(t)$:
\begin{equation}
\label{equation-t}
\begin{split}
&\begin{cases}
(\hat t{\rho}{u})'=0\\
(\hat{t}{\rho}{u}^2)'+\hat t{p}'=\hat t {\rho}{E}\\ (\hat t {\rho}{u}{B})'=\hat t{\rho}{u}{E}\\
(\hat t{E})'=\hat t({\rho}-{b})
\end{cases}\quad\tx{for}\,\,0<t<r_0-r_1(:=T),\\
&({\rho}, {u}, {p}, {E})(0)=(\rho_0, u_0, p_0, E_0),
\end{split}
\end{equation}
with
\begin{equation}
\label{defi-Bernoulli}
B=\frac 12 u^2+\frac{\gam p}{(\gam-1)\rho},
\end{equation}
where we denote $\frac{d}{dt}$ by $'$.


The R-H conditions \eqref{rhb}--\eqref{rh-Phi} for radial solutions of \eqref{ep} in terms of $(\rho, u, p, E)$ become
\begin{align}\label{rhr}
[\rho u]_{\Gam_s}=[\rho u^2+p]_{\Gam_s}=[B]_{\Gam_s}=[E]_{\Gam_s}=0,
\end{align}
where $\shock$ is given as
\begin{equation*}
\shock=\{t=t_s\}\quad\tx{for some $t_s\in (0, T)$.}
\end{equation*}


From the second and third equation of \eqref{equation-t}, one can directly derive that $S$ from \eqref{2d-1-b1} satisfies
\begin{align}\label{sp}
S'=0.
\end{align}
This and the first equation in \eqref{equation-t} yield
\begin{align}\label{m0s0}
(\hat{t}\rho u,e^S) =(m_0,\kappa_0)\qd\tx{on}\,\,[0, T] \quad\tx{for}\,\,
(m_0,\kappa_0)=(r_0\rho_0u_0,\,\,\frac{p_0}{\rho_0^{\gam}}).
\end{align}
From \eqref{m0s0} and the definition $M^2=\frac{u^2}{\gam p/\rho}$, it is directly derived that
\begin{equation}
\label{cToM}
c^2=\gam \kappa_0\rho^{\gam-1}=
\gam\kappa_0\left(\frac{m_0^2}{\gam\kappa_0}\right)^{\frac{\gam-1}{\gam+1}}
\left(\frac{1}{\hat t^2 M^2}\right)^{\frac{\gam-1}{\gam+1}}=:\frac{1}{\mu_0}\left(\frac{1}{\hat t^2 M^2}\right)^{\frac{\gam-1}{\gam+1}}.
\end{equation}
By \eqref{m0s0}, the third equation in \eqref{equation-t} can be simplified as
\begin{align}
\label{bp} B^\p=E.
\end{align}
By  $\eqref{m0s0}$  and \eqref{cToM}, equation $\eqref{bp}$ can be written as a nonlinear equation for $(M,E)$ as follows:
\begin{align}\label{odeM}
(M^2)'=\frac{M^2}{(M^2-1)}\left((\gam+1)\mu_0 E (\hat t^2M^2)^{\frac{\gam-1}{\gam+1}}-\frac{1}{\hat t}(2+(\gam-1)M^2)\right)
=:h_1(t, M, E, \kappa_0).
\end{align}
We solve \eqref{cToM} for $\rho$ to get
\begin{equation}
\label{rM-expression}
\rho=\left(\frac{1}{\gam\kappa_0\mu_0}\right)^{\frac{1}{\gam-1}}\left(\frac{1}{\hat t^2M^2}\right)^{\frac{1}{\gam+1}}
=:\mu_1\left(\frac{1}{\hat t^2M^2}\right)^{\frac{1}{\gam+1}},
\end{equation}
then substitute this expression into the last equation in \eqref{equation-t} to get
\begin{equation}
\label{new-equation-E}
(\hat{t}E)'=\hat t\left(\mu_1\left(\frac{1}{\hat t^2M^2}\right)^{\frac{1}{\gam+1}}-b(t)\right)=:h_2(t, M, E, \kappa_0).
\end{equation}
Note that the constants $\mu_0$ and $\mu_1$ vary depending on $(\rho_0, u_0, p_0)$, but they are independent of $E_0$.

By \eqref{m0s0}, \eqref{odeM} and \eqref{new-equation-E},  if $M\neq 1$ and $\rho u>0$ for $0<t<T$, then \eqref{equation-t} is equivalent to
\begin{equation}
\label{system-simplified1}
\begin{cases}
(\hat{t}\rho u,e^S) =(m_0,\kappa_0),\\
(M^2)'=h_1(t, M, E,\kappa_0),\\
(\hat t E)'=h_2(t, M, E, \kappa_0),
\end{cases}
\end{equation}
which consists of two algebraic equations and a first order nonlinear ODE system for $(M, E)$.
In \eqref{odeM} and \eqref{new-equation-E}, we represent $(h_1,h_2)$ as functions varying with respect to $\kappa_0$. This is because, we will consider the system \eqref{system-simplified1} for different values of $\kappa_0$, while $m_0$ is fixed same always due to the R-H conditions \eqref{rhr}.

For later use, we also note that \eqref{system-simplified1} is equivalent to
\begin{equation}
\label{rd}
\begin{cases}
(\hat{t}\rho u,e^S) =(m_0,\kappa_0),\\
\rho'=g_1(t, \rho, E, \kappa_0),\\
(\hat t E)'=g_2(t, \rho, E),
\end{cases}
\end{equation}
for
\begin{equation}
\label{def-gs}
(g_1, g_2)(t, \rho, E, \kappa_0)=\left(\frac{\rho(\hat tE-\frac{m_0^2}{\hat t^2\rho^2})}{\hat t(\gam \kappa_0\rho^{\gam-1}-\frac{m_0^2}{\hat t^2\rho^2})}, \hat t(\rho-b(t))\right).
\end{equation}

Fix $\gam>1$, and suppose that the initial condition $({\rho}, {u}, {p}, {E})(0)=(\rho_0, u_0, p_0, E_0)$ in \eqref{equation-t} satisfies $M_0^2:=\frac{u_0^2}{(\gam p_0/\rho_0)}>1$. And, suppose that \eqref{equation-t} has a $C^1$ solution $(\rho_-, u_-, p_-, E_-)$ on the interval $[0,T]$ with
\begin{equation}
\label{assumption-incsol}
\rho_->0,\quad u_->0,\quad \tx{and} \quad M_-=\frac{u_-}{\sqrt{\gam \kappa_0\rho_-^{\gam-1}}} >1\quad\tx{for}\,\,0\le t \le T.
\end{equation}
The unique existence of such a $C^1$ solution $(\rho_-, u_-, p_-, E_-)$ to  \eqref{equation-t} is discussed in the next section.

For a fixed constant $t_s(=:r_0-r_s, r_s\in(r_1, r_0))\in (0, T)$, we construct a radial shock solution  $(\rho, u, p, E)(t;t_s)$ of \eqref{ep} with incoming state $(\rho_-, u_-, p_-, E_-)$, and with a shock $\shock=\{t=t_s\}$ in the sense of Definition \ref{definition-shock-sol}.
For each $t_s\in(0,T)$, the shock solution $(\rho, u, p, E)(t;t_s)$ is represented as
\begin{equation}
\label{shocksol-rep}
(\rho, u, p, E)(t;t_s)=
\begin{cases}
(\rho_-, u_-, p_-, E_-)(t)\quad\tx{for $0\le t<t_s$}\\
(\rho_+, u_+, p_+, E_+)(t;t_s)\quad\tx{for $t_s\le t\le T$}
\end{cases},
\end{equation}
for $(\rho_+, u_+, p_+, E_+)(t;t_s)$ defined as follows:
\begin{itemize}
\item[(1)] Set $(\rho_s,u_s,p_s,E_s):=(\rho_+,u_+,p_+,E_+)(t_s;t_s)$. A direct computation using \eqref{rhr} yields
\begin{align}\label{rhc}
\begin{array}{l}
(\rho_s, u_s, p_s, E_s)
=(\frac{\rho_-u_-^2}{K_-}, \frac{K_-}{u_-}, \rho_- u_-^2+p_--\rho_-K_-, E_-)
\end{array}
\end{align}
for $K_-=\frac{2(\gam-1)}{\gam+1}(\frac 12 u_-^2+\frac{\gam p_-}{(\gam-1)\rho_-})$ where $(\rho_-,u_-,p_-,E_-)$ are evaluated at $\shock=\{t=t_s\}$.

\item[(2)]
Set
\begin{equation}
\label{definition-kappa}
\kappa_s:=\frac{p_s}{\rho_s^{\gam}},
\end{equation}
and solve
\begin{equation}
\label{system-downstream}
\begin{cases}
(\hat{t}\rho u,e^S) =(m_0,\kappa_s)\\
(M^2)'=h_1(t, M^2, E,\kappa_s)\\
(\hat t E)'=h_2(t, M^2, E, \kappa_s)
\end{cases},\quad\tx{or equivalently}\quad
\begin{cases}
(\hat{t}\rho u,e^S) =(m_0,\kappa_s)\\
\rho'=g_1(t, \rho, E, \kappa_s)\\
(\hat t E)'=g_2(t, \rho, E)
\end{cases}
\end{equation}
with initial condition \eqref{rhc}
for $t_s<t\le T$. If \eqref{system-downstream} has a unique $C^1$ solution with $M\neq 1$ for $t_s\le t\le T$, we denote the solution by $(\rho_+, u_+, p_+, E_+)(t;t_s)$ .
\end{itemize}
By the assumption \eqref{assumption-incsol}, we get from \eqref{rhc} that $(\rho_s, u_s, p_s)$ are strictly positive, and that
\begin{equation}
\label{inequality-rhadm}
0<u_s(t_s)<u_-(t_s),\quad\tx{and}\quad M_s^2(t_s)-1=-\frac{p_-(t_s)}{p_s(t_s)}(M_-^2(t_s)-1)<0,
\end{equation}
for $(M^2_s, M^2_-)(t_s):=(\frac{u_s^2}{(\gam p_s/\rho_s)}, \frac{u_-^2}{(\gam p_-/\rho_-)})(t_s)$. This indicates that, for each $t_s\in(0, T)$,
the shock solution $(\rho, u, p, E)(t;t_s)$ is {\emph{admissible}} in the sense of Definition \ref{definition-admissibility}, and that $\shock=\{t=t_s\}$ is a {\emph{transonic shock}} in the sense of Definition \ref{definition-transonic-shock}.

Our main interest is in the radial shock solutions of \eqref{ep} which behave similar to the ones of Euler system in a divergent nozzle(Figure \ref{figure2}). Therefore, we define  a particular class of shock solutions as follows:
\begin{definition}[Radial transonic shock solution of \eqref{ep} with positive direction of electric field]
\label{definition-admissible-sol}
For a fixed $t_s\in(0, T)$, we define a shock solution $(\rho, u, p, E)(t;t_s)$ to be a {\emph{radial transonic shock solution of \eqref{ep} with positive direction of electric field}} if it satisfies the following conditions:
\begin{itemize}
\item[(i)] $(\rho, u, p, E)(t;t_s)$ are strictly positive for $0\le t\le T$;
\item[(ii)] the Mach number $M_+(t;t_s)=\frac{u_+(t;t_s)}{\sqrt{(\gam p_+(t;t_s)/\rho_+(t;t_s))}}$ downstream satisfies
\begin{equation}
\label{condition-subsonic}
M_+^2(t;t_s)<1\quad\quad\tx{for $t_s\le t\le T$.}
\end{equation}
\end{itemize}
\end{definition}

Our main theorem is stated below.

\begin{thm} \label{thmshock}
Fix two constants $\gam>1$ and $r_0>0$. And, fix a $C^1$ function $b(t)(=\til b(r))$ with $b(t)>0$ for $t\in[0,r_0]$ and $\|b\|_{C^1([0, r_0])}\le b_0$ for some constant $b_0>0$. Fix positive constant data $(\rho_0, u_0, p_0)$ with $M_0(=\sqrt{\frac{u_0^2}{\gam p_0/\rho_0}})>1$.
\begin{itemize}
\item[(a)] There exists $\underline{E}>0$ depending on $(\gam, r_0, b_0, \rho_0, u_0, p_0)$ such that whenever $E_0\ge \underline{E}$, a family of transonic shock solution $(\rho, u, p, E)(t;t_s)$ in the form of \eqref{shocksol-rep} is uniquely given on the interval $I_T=[0, T]$ with satisfying
\begin{equation}
\label{tshock-ic}
(\rho, u, p, E)(0)=(\rho_0, u_0, p_0, E_0),
\end{equation}
and
\begin{equation}
\label{p-monot}
	\frac{dp_+}{dt_s}(T;t_s)<0\quad\tx{for all $t_s\in[0, T]$},
	\end{equation}
where $T$ is sufficiently small depending on $(\gam, r_0, b_0, \rho_0, u_0, p_0, E_0)$.	
\item[(b)] For $\gam\ge 2$, if $r_1\in(0, r_0)$ satisfies $\ln \frac{r_0}{r_1}<\frac{\gam+1}{2(\gam-1)}$, then there exists $E_*>0$ depending on $(\gam, r_0, r_1, b_0, \rho_0, u_0, p_0)$ such that whenever $E_0\ge E_*$, a family of transonic shock solution $(\rho, u, p, E)(t;t_s)$ is uniquely given on the interval $I_{T_*}:=[0, r_0-r_1]$ with satisfying \eqref{tshock-ic} and \eqref{p-monot}.

\end{itemize}
\end{thm}

\begin{remark}
  According to Theorem \ref{thmshock}, given incoming supersonic radial flow of Euler-Poisson system in an annular domain $\nn$, a transonic shock problem with a fixed subsonic pressure on the inner boundary $\Gam_{\rm ex}$ of $\nn$(see Figure \ref{figure1}) is {\emph{conditionally well-posed}} in the sense that if the magnitude of the electric field is sufficiently strong, then there exists a unique radial transonic shock solution of \eqref{ep}.
\end{remark}

The following proposition is the key ingredient to prove Theorem \ref{thmshock}:
\begin{prop}\label{promp}
Suppose that a shock solution $(\rho, u, p, E)(t;t_s)$ in \eqref{shocksol-rep} is a {\emph{radial transonic shock solution with positive direction of electric field}} for all $t_s\in[0, T]$, in the sense of Definition \ref{definition-admissible-sol}. In addition, assume that
\begin{itemize}
\item[(i)] $M_-^2=\frac{u_-^2}{(\gam p_-/\rho_-)}$ satisfies
\begin{equation}
\label{condition-M-mt}
\frac{dM_-^2}{dt}>0\quad\tx{for $0 < t < T$};
\end{equation}
\item[(ii)] for each $t_s\in(0, T)$, $\rho_+(t;t_s)$ satisfies
\begin{equation}
\label{condition-rho-mt}
\frac{d\rho_+}{dt}(t;t_s)>0\quad\tx{for all $t_s < t < T$}.
\end{equation}
\end{itemize}
Then, we have
\begin{equation*}
\frac{dp_+}{dt_s}(T;t_s)<0\quad\tx{for all $t_s\in(0, T)$}.
\end{equation*}
\end{prop}
Before proving Proposition \ref{promp}, a few preliminary lemmas need to come first.
\begin{lem} \label{lemS}
Under the same assumptions as Proposition \ref{promp}, $\kappa_s=\kappa_s(t_s)$ defined by \eqref{definition-kappa} satisfies
\begin{align}\label{sprove}
\frac{d \kappa_s(t_s)}{d t_s}>0\quad\tx{for all $t_s\in(0,T)$.}
\end{align}

	\begin{proof}
		By \eqref{rhc} and \eqref{definition-kappa}, we get
		$
		\kappa_s(t_s)=\left.\frac{\rho_-u_-^2+p_--\rho_-K_-}{\left({\rho_-u_-^2}/{K_-}\right)^\gam}\right|_{t=t_s}.
		$
		By using the definition $M_-^2=\frac{u_-^2}{(\gam p_-/\rho_-)}$, we rewrite $\kappa_s$ as
		\begin{align}
		\label{sMain}\kappa_s(t_s)=f(M_-^2(t_s))\kappa_0,
		\end{align}
		for
		\begin{align*}
		f(x)=\frac{1}{\gam+1}(2\gam x-(\gam-1))\left(\frac{\gam-1}{\gam+1}+\frac{2}{\gam+1}\frac{1}{x}\right)^\gam.
		\end{align*}
		Note that $f(1)=1$, and
		\begin{align}
		\label{fdiff}f^\p(x)=\frac{2\gam(\gam-1)}{(\gam+1)^2}\left(\frac{\gam-1}{\gam+1}+\frac{2}{\gam+1}\frac{1}{x}\right)^{\gam-1}\left(\frac{1}{x}-1\right)^2>0\qd \tx{for} \qd x>0.
		\end{align}
		Since $f(M_-^2(t_s))>1$ for all $t_s\in(0, T)$ by \eqref{assumption-incsol}, we obtain from
		\eqref{sMain} that
		$$
		\kappa_s(t_s)>\kappa_0\quad\tx{for all $t_s\in(0, T)$}.
		$$
		To prove $\eqref{sprove}$, we differentiate $\eqref{sMain}$ with respect to $t_s$ to get
		\begin{align}
		\label{sdiff}
		\frac{d\kappa_s(t_s)}{d t_s}=f' (M_-^2(t_s))\frac{d M_-^2(t_s)}{d t_s}\kappa_0.
		\end{align}
Then, \eqref{sprove} is directly obtained from  \eqref{condition-M-mt} and \eqref{fdiff}.	
			
\end{proof}
\end{lem}

Under the assumption of \eqref{assumption-incsol}, $(\rho_s, u_s, p_s, E_s)$ are $C^1$ with respect to $t_s\in(0, T)$. And, $(h_1, h_2)(t, M^2, E, \kappa_s)$ (or, $(g_1, g_2)(t, \rho, E, \kappa_s)$) are $C^1$ with respect to $(t, M^2, E, \kappa_s)$ (or, $(t, \rho, E, \kappa_s)$), for $(h_1, h_2, g_1,g_2)$ defined by \eqref{odeM}, \eqref{new-equation-E}, \eqref{def-gs}. Therefore, $(\rho_+, u_+, p_+, E_+ )(t;t_s)$ are $C^1$ with respect to $t_s\in(0, t]$ for each $t\in(0, T)$.
\begin{lem}
	\label{lemrho}
	Under the same assumptions as Proposition \ref{promp}, for each fixed $\bar t\in(0, T)$, $\rho_+(\bar t;t_s)$ satisfies
	\begin{align*}
	\frac{\partial\rho_+}{\partial t_s}(\bar t;t_s)<0\qd\tx{for $t_s\in(0, \bar t]$}.
	\end{align*}
	\begin{proof}
	In this proof, we use the second formulation in \eqref{system-downstream}.
	
(Step 1)	Fix $\bar t\in (0,T]$. For each $t_s\in (0, \bar t]$, we set as
	\begin{equation*}
	\mathfrak{g}_1(t;t_s):=g_1(t, \rho_+(t;t_s), E_+(t;t_s), \kappa_s(t_s)).
	\end{equation*}
	By integrating $\frac{d}{d\eta}\rho_+(\eta;t_s)=\mathfrak{g}_1(\eta;t_s)$ with respect to $\eta$ over the interval $[t_s, t]$ for some $t\in[t_s, \bar t]$, we obtain that
	\begin{align}\label{rho+}
		\rho_{+}\left( t;t_{s}\right)  =\rho_{s}(t_s)+\int_{t_{s}}^{t} \mathfrak{g}_1(\eta;t_s)\,d\eta		
		\end{align}
		where $\rho_s$ is given in $\eqref{rhc}$. Take the partial derivative of $\rho_+(t;t_s)$ with respect to $t_s$ to get		\begin{equation}\label{rho+rs}
		\begin{split}
		\frac{\partial \rho_{+}}{\partial t_{s}}\left(  t;t_s\right)
		=
		&\frac
		{d\rho_{s}(t_s)}{d t_{s}}-g_1(t_s, \rho_s(t_s), E_s(t_s), \kappa_s(t_s))\\
		&+\int_{t_{s}}^{t} \frac{\partial g_1(\eta, \rho_+(\eta; t_s), E_+(\eta; t_s), \kappa_s(t_s))}
		{\partial{(\rho, E, \kappa_s)}}		
		\cdot \frac{\partial}{\partial t_s}(\rho_+(\eta; t_s), E_+(\eta; t_s), \kappa_s(t_s))\,d\eta.
		\end{split}
		\end{equation}
		A direct computation using \eqref{def-gs} shows that
		\begin{equation*}
		\begin{split}
		&\frac{\partial g_1( t,\rho_+, E_+, \kappa_s)}
		{\partial{(E, \kappa_s)}}=\frac{(\rho_+^{2-\gam}, -\gam\rho_+')}{\gam\kappa_s(1-M_+^2)}=:(a_1,a_2)(t;t_s),\\
		&\frac{\partial g_1( t,\rho_+, E_+, \kappa_s)}
		{\partial{\rho}}=(\ln \rho_+)'\frac{(2-\gam)-3M_+^2}{1-M_+^2}+\frac{2M_+^2}{\hat t(1-M_+^2)}=:a_3(t;t_s).
		\end{split}
\end{equation*}
By the conditions \eqref{condition-subsonic} and \eqref{condition-rho-mt}, one can choose a positive constant $\lambda_0>1$ depending on $(\gam, \rho_0, u_0, p_0, E_0, T)$ to satisfy
		\begin{equation}\label{abc}
		\begin{split}
		&\frac{1}{\lambda_0}\le \frac{\partial g_1( t,\rho_+, E_+, \kappa_s)}{\partial E}, -\frac{\partial g_1(t, \rho_+, E_+, \kappa_s)}{\partial\kappa_s}\le \lambda_0,\\
		&|\frac{\partial g_1( t, \rho_+, E_+, \kappa_s)}
		{\partial{\rho}}|\le \lambda_0
		\end{split}
		\end{equation}
		for all $t_s\in[0, T]$ and $t\in[t_s, T]$. One can further adjust $\lambda_0>0$ so that Lemma $\ref{lemS}$ yields
		\begin{align}\label{d}
			\frac{d \kappa_s(t_s)}{d t_s}    \ge \frac{1}{\lambda_0} \quad\tx{for all $t_s\in[0, T]$}.
		\end{align}
		
(Step 2) Set $X(t;t_s):=\frac{\partial \rho_+}{\partial t_s}(t;t_s)$. By $\eqref{rho+rs}$, $X$ becomes a solution to
		\begin{equation}\label{xr}
		\frac{d X}{d t}=a_3X+a_2\frac{d\kappa_s(t_s)%
		}{d t_{s}}+a_1\frac{\partial E_{+}(t;t_s)}{\partial t_{s}},
		\end{equation}
		A direct computation using \eqref{def-gs}, \eqref{assumption-incsol} and \eqref{rhc} yields that
		\begin{equation}
		\label{def-Xs}
		\begin{split}
		X(t_s;t_s)
		&=\frac{d\rho_{s}(t_s)}{d t_{s}}-g_1(t_s, \rho_s(t_s), E_s(t_s), \kappa_s(t_s))\\
		&=\left.-\frac{2\gam^2\kappa_0}{(\gam+1)^2}\frac{\rho_-^{\gam}(t_s)}{u_s^2(t_s)} \frac{M_-^2-1}{M_-^2} (M_-^2)'\right|_{t=t_s}<0\quad\tx{for all $t_s\in[0, T]$.}
		\end{split}
\end{equation}
Since $X(t;t_s)$ is $C^1$ for $t\in[t_s, T]$, there exists a small constant $\varepsilon>0$ such that $X(t;t_s)<0$ for all $t\in[t_s, t_s+\varepsilon]$.
Let $t^*\in(t_s, \bar{t}]$ be the smallest value of $t$ such that $X(t;t_s)<0$ for $t<t^*$, and $X(t;t_s)\ge 0$ for $t>t^*$. Then,
\begin{equation}
\label{def-tstar}
X(t^*;t_s)=0.
\end{equation}

(Step 3) For further estimate of $X(t;t_s)$, we get back to the equation $(\hat t E)'=g_2(t,\rho, E)$ in \eqref{system-downstream}. We integrate this equation with respect to $t$ over the interval, then take the partial derivative of the resultant equation with respect to $t_s$ to get, 		
\begin{align}\label{Ers}
		\frac{\pt E_+}{\pt t_s}(t;t_s)
		=\frac{1}{r_0-t}\left((r_0-t_s)(\rho_--\rho_s)+\int_{t_s}^t(r_0-\eta)X(\eta;t_s)\,  d\eta\right).
		\end{align}
		From this, we obtain that
		\begin{align}\label{Ers}
		(\hat t Y)'(t;t_s)=\hat t X(t;t_s), \quad Y(t_s;t_s)=(\rho_--\rho_s)(t_s)		\quad\tx{for $Y(t;t_s):=\frac{\pt E_+}{\pt t_s}(t;t_s)$.}
		\end{align}
		By  step 2, we get from \eqref{rhc}, \eqref{inequality-rhadm} and \eqref{Ers} that
		\begin{equation}
			\label{e1}
		\frac{\pt E_+}{\pt t_s}(t;t_s)\le (\rho_--\rho_s)(t_s)<0\quad\tx{for all $t\in[t_s, t^*]$.}		
		\end{equation}

By $\eqref{abc}$--$\eqref{xr}$ and $\eqref{e1}$, there exists a constant $\lambda_1>0$ such that
\begin{equation*}
\frac{dX}{dt}-a_1X\le -\lambda_1\quad\tx{for $t\in[t_s, t^*]$.}
\end{equation*}
By the simple method of integrating factor, we obtain that
\begin{equation*}
\mu(t^*)X(t^*;t_s)\le X(t_s;t_s)-\lambda_1\int_{t_s}^{t^*} \mu(\eta)\,d\eta\le X(t_s;t_s)<0
\end{equation*}
for $\mu(\eta):=\exp(-\int_{t_s}^{\eta} a_1(\til{\eta};t_s)\,d\til{\eta})$. But this contradicts to \eqref{def-tstar}. Therefore, we finally conclude that $X(\bar t; t_s)=\frac{\der\rho_+}{\der t_s}(\bar t;t_s)<0$.
This finishes the proof.
	\end{proof}
\end{lem}

We are now ready to prove the monotonicity of the exit pressure $p_+(T;t_s)$ with respect to $t_s$.

	\begin{proof}[Proof of Proposition \ref{promp}]
	For $t_s\in[0, T]$,
	define
	\begin{equation*}
	B_{+}\left(T;t_{s}\right)
	=(\frac{u_{+}^{2}}{2}+\frac{\gam p_+}{(\gam-1)\rho_+})(T;t_s).
	\end{equation*}
	Substituting the expressions $u_+(t;t_s)=\frac{m_0}{\hat t\rho_+(t;t_s)}$ and
	$\rho_{+}(t;t_s)=\left(  \frac{p_{+}(t;t_s)}{\kappa_s(t_s)}\right) ^{\frac{1}{\gam}}$ into the definition stated above, we get		\begin{equation}
	B_{+}\left(T;t_{s}\right)
	=\frac{1}{2}\left(  \frac{m_{0}}{r_1}\right)^{2}\left(  \frac{\kappa_s(t_s)}{p_{+}(T;t_s)}\right) ^{\frac{2}{\gam}}+\frac{\gam}{\gam-1}p_{+}^{1-\frac{1}{\gam}}(T;t_s)\kappa_s^{\frac{1}{\gam}}(t_s)=:G\left(p_{+}(T;t_s),\kappa_s(t_s)\right).
	\label{Beqn}%
	\end{equation}
	We differentiate \eqref{Beqn} with respect to $t_{s}$ to get
	\begin{equation}
	\label{B-expression}
	\begin{split}
	\frac{d}{d t_{s}}B_{+}\left( T;t_{s}\right)
	&=G_{p_{+}}(p_+(T;t_s), \kappa_s(t_s))
	\frac{dp_{+}}{d t_{s}}(T;t_s)+G_{\kappa_s}(p_+(T;t_s), \kappa_s(t_s))\frac{d\kappa_{s}}{d
		t_{s}}(t_s)
	\end{split}
	\end{equation}
	with
	\begin{equation*}
	(G_{p_+}, G_{\kappa_s})(p_+(T;t_s), \kappa_s(t_s))
	=(\frac{1-M_+^2(T;t_s)}{\rho_+(T;t_s)},  \,\,
	 \frac{m_0^2\kappa_s^{\frac{2}{\gam}-1}(t_s)}{r_1^2p_{+}^{\frac{2}{\gam}}(T;t_s)}+\frac{1}{\gam-1}
	\left(\frac{p_{+}(T;t_s)}{\kappa_s(t_s)}\right)^{1-\frac{1}{\gam}} ).
	\end{equation*}
From  \eqref{rhr} and \eqref{bp}, it follows that 
	\begin{equation}
	\label{B+}
	B_{+}\left(  T;t_{s}\right)  =B_-\left(  t_{s}\right) +
	\int_{t_{s}}^{T}E_{+}\left(  \eta;t_s\right)\,
	d\eta.
	\end{equation}
By differentiating (\ref{B+}) with respect to $t_{s}$, and by applying \eqref{rhc}, one gets
	\begin{equation*}
	\frac{d}{d t_{s}}B_{+}\left(  T;t_{s}\right)  =\int_{t_{s}}
	^{T}\frac{\pt E_{+}  }{\pt t_{s}}\left(  \eta;t_{s}\right)\,d\eta.
	\end{equation*}
	Due to Lemma \ref{lemrho} and \eqref{Ers}, it holds that $\frac{\pt E_{+} }{\pt t_{s}}\left(
	t;t_{s}\right) <0$ for all $t\in\left[  t_s,T\right]$, and this implies that
	\begin{equation*}
	\frac{d}{d t_{s}}B_{+}\left(  T;t_{s}\right)<0.
	\end{equation*}
Finally, combining this with \eqref{condition-subsonic} and \eqref{B-expression} implies \eqref{p-monot}.
\end{proof}


\section{Supersonic flow}\label{secIVP}
Before proving Theoren \ref{thmshock}, we first prove the existence of radial supersonic solutions of \eqref{ep} satisfying \eqref{condition-M-mt}.


\begin{prop}
\label{proposition-supersonic}

Fix $\gam>1$ and $r_0>0$ for the definition of the annulus $\nn$ in \eqref{def-annulus}. Given positive constant data $(m_0, \kappa_0, M_0)$ with $M_0>1$, if $r_1\in(0, r_0)$ in \eqref{def-annulus} satisfies
\begin{equation}
\label{condition-ann-ration}
\ln \frac{r_0}{r_1}<\frac{\gam+1}{2(\gam-1)},
\end{equation}
then there exists a constant $\underline{E}>0$ depending on $(\gam, m_0, \kappa_0, M_0, r_0, r_1, b_0)$ so that whenever $E_0 > \underline{E}$, the initial value problem \eqref{system-simplified1} with $(M, E)(0)=(M_0, E_0)$ has a unique $C^1$ solution $(M, E)(t)$ for $t\in[0, r_0-r_1]$ satisfying \eqref{condition-M-mt}.
\begin{proof}
(Step 1)
We rewrite $\eqref{odeM}$ as
\begin{equation}
\label{equation-M-another}
(M^2)'=(\gam+1)\mu_0 \hat t^{\frac{\gam-3}{\gam+1}}\frac{(M^2)^{\frac{2\gam}{\gam+1}}}{M^2-1}\left(\hat t E-K\right),
\end{equation}
for $K=\frac{2(\gam-1)}{\gam+1}B$, where $B$ is defined by \eqref{defi-Bernoulli}. Since $(h_1, h_2)$ from \eqref{odeM} and \eqref{new-equation-E} are $C^1$ with respect to $(t, M^2, E)$ for $M\neq 1$, the initial value problem \eqref{system-simplified1} with $(M, E)(0)=(M_0, E_0)$ with $M_0>1$ and $E_0>0$ has a unique $C^1$ solution for $t\in[0, \varepsilon_0)$ for some small constant $\varepsilon_0>0$ with $M^2(t)>1$ for all $t\in[0, \varepsilon_0)$.

(Step 2)
By the definition of $K$ and \eqref{bp}, $K$ satisfies
\begin{equation}\label{kprime}
K'=\frac{2(\gam-1)}{\gam+1}E.
\end{equation}
Combine this equation with \eqref{new-equation-E} to get
\begin{align}
\label{rekdiff}
(\hat t E- K)^\p=\hat t\left(\mu_1\left(\frac{1}{\hat t^2M^2}\right)^{\frac{1}{\gam+1}}-b(t)\right)-\frac{2(\gam-1)}{\gam+1}E.
\end{align}
From \eqref{new-equation-E} and $M^2>1$, it follows that
\begin{align}
\label{E-Sup-ineq-1}
		-\hat t b_0<(\hat t E)^\p<\hat t \mu_1\left(\frac{1}{\hat t}\right)^{\frac{1}{2(\gam+1)}}
		\qd\tx{for all $t\in [0, \varepsilon_0)$}.
		\end{align}
By integrating this inequality over the interval $[0, t]$ for $t\in(0, \varepsilon_0)$, we obtain that
		\begin{align}\label{E-Sup-ineq-2}
		\frac{r_0E_0-\alp_1}{r_0-t}<E(t)<\frac{1}{r_0-t}\left(r_0E_0+\int_0^t (r_0-\eta)
		\mu_1\left(\frac{1}{r_0-\eta }\right)^{\frac{2}{(\gam+1)}}\,d\eta\right)\le
		\frac{r_0E_0+\alpha_1}{r_0-t},
		\end{align}
		for a constant $\alpha_1>0$ depending only on $(r_0,  \kappa_0, m_0, b_0, \gam)$. Note that $\alpha_1$ is independent of $\varepsilon_0$ and $E_0$. Combine \eqref{E-Sup-ineq-2} with \eqref{rekdiff} to get
		\begin{align*}
		(\hat t E- K)(t)&> r_0E_0-K_0
		 -\int_0^t r_0b_0+\frac{2(\gam-1)(r_0E_0+\alpha_1)}{(\gam+1)(r_0-\eta)}\,d\eta\\
		 &= r_0E_0(1-\frac{2(\gam-1)}{\gam+1}\ln \frac{r_0}{r_0-t})-K_0-\int_0^t r_0b_0+\frac{2(\gam-1)\alpha_1}{(\gam+1)(r_0-\eta)}\,d\eta \\
		 &\ge r_0E_0(1-\frac{2(\gam-1)}{\gam+1}\ln \frac{r_0}{r_0-\varepsilon_0})-\alpha_2\qd\tx{ for $t\in [0,\varepsilon_0]$},
		\end{align*}
for $K_0:=\frac{\gam-1}{\gam+1}u_0^2+\frac{2\gam p_0}{(\gam+1)\rho_0}$ and a constant $\alpha_2>0$ depending on $(r_0, r_1, M_0, \kappa_0, m_0, \gam,b_0)$ but independent of $(\varepsilon_0, E_0)$.

(Step 3) Fix a constant $r_1>0$ to satisfy \eqref{condition-ann-ration}, and set
\begin{equation*}
\delta_0:=1-\frac{2(\gam-1)}{\gam+1}\ln \frac{r_0}{r_1}
\end{equation*}
so that $0<\delta_0<1$ holds. Then we have
$r_0E_0(1-\frac{2(\gam-1)}{\gam+1}\ln \frac{r_0}{r_0-\varepsilon_0})-\alpha_2\ge r_0E_0\delta_0-\alp_2$, whenever $\varepsilon_0\in(0, r_0-r_1).$ Set $\underline{E}$ as
\begin{equation*}
\underline{E}:= \frac{\alp_2}{r_0\delta_0}.
\end{equation*}
Note that $\underline{E}$ depends only on $(\gam, m_0, \kappa_0, M_0, r_0, r_1, b_0)$.
If $E_0>\underline{E}$, then we obtain that
\begin{equation}
\label{M-mtprop-si}
(M^2)'(t)>0\quad\tx{for all $t\in[0, \varepsilon_0)$.}
\end{equation}

(Step 4)
Given positive constant data $(m_0, \kappa_0, M_0, E_0)$ with $M_0>1$ and $E_0>\underline{E}$, we have shown in (step 1)--(step 3) that the initial value problem \eqref{system-simplified1} with $(M, E)(0)=(M_{0}, E_0)$ has a unique smooth solution on the interval $[0, \varepsilon_0)$ for some small constant $\varepsilon_0>0$, furthermore the solution satisfies $(M^2)'>0$ on $(0, \varepsilon_0)$. It remains to extend the solution up to $t=r_0-r_1(=:T)$.

By \eqref{E-Sup-ineq-1} and \eqref{E-Sup-ineq-2}, $E$ can be extended up to $t=\varepsilon_0$ as
\begin{equation*}
(r_0-\varepsilon_0)E(\varepsilon_0)=r_0E_0+\lim_{t\to\varepsilon_0-}\int_0^{t} h_2(\eta, M(\eta), E(\eta), \kappa_0)\,d\eta,
\end{equation*}
there exists a constant $\alpha_3>0$ depending on $(\gam, m_0, \kappa_0, M_0, E_0, b_0, r_0, r_1)$, but independent of $\varepsilon_0$ such that
\begin{equation*}
\sup_{t\in[0, \varepsilon_0]}|E(t)|\le \alpha_3.
\end{equation*}
Then \eqref{odeM} yields that
\begin{equation*}
0<\frac{(M^2)'}{(M^2)^{\frac{\gam-1}{\gam+1}}}\le  \frac{M_0^2}{M_0^2-1}(\gam+1)\alpha_3
\frac{(r_0^2)^{\frac{\gam-1}{\gam+1}}}{r_1}\quad\tx{on $(0, \varepsilon_0)$.}
\end{equation*}
By integrating the inequality above over the interval $[0, t)$ for $t\in(0, \varepsilon_0)$, we obtain that
\begin{equation*}
(M_0^2)^{\frac{2}{\gam+1}}<(M^2(t))^{^{\frac{2}{\gam+1}}}\le \frac{2t}{\gam+1}\frac{M_0^2}{M_0^2-1}(\gam+1)
\alpha_3\frac{(r_0^2)^{\frac{\gam-1}{\gam+1}}}{r_1}\quad\tx{for $t\in (0,\varepsilon_0)$}.
\end{equation*}
By \eqref{M-mtprop-si} and the monotone convergence theorem, the limit $\underset{t\to \varepsilon_0-}{\lim}M^2(t)$ exists, therefore $M^2$ can be extended up to $t=\varepsilon_0$. Then, one can repeat the argument in (step 1)--(step 3) to conclude that the $C^1$ solution to \eqref{system-simplified1} and $(M, E)(0)=(M_{0}, E_0)$ with $M_0>1$ and $E_0>\underline{E}$ uniquely exists up to $t=r_0-r_1$, and the solution satisfies \eqref{condition-M-mt} for all $t\in(0, r_0-r_1]$.

\end{proof}

\end{prop}

\begin{remark}
In the definition \eqref{def-annulus} of $\mcl{A}$, denote $r_1$ by $r_1=\lambda r_0$ for $\lambda\in(0,1)$. According to Proposition \ref{proposition-supersonic}, if the convergent ratio of the annulus $\lambda(=\frac{r_1}{r_0})$ is greater than $e^{-\frac{\gam+1}{2(\gam-1)}}$, then a radial supersonic solution to \eqref{ep} in $\nn$ is \emph{relatively accelerating} throughout the annulus $\nn$, in the sense that the Mach number $M$ monotonically increases in the flow direction, provided that the magnitude of electric field ${\bf E}=-E_0\hat{{\bf r}}$ at the entrance $\Gam_0$ is sufficiently large depending on $(\gamma, m_0, \kappa_0, M_0, r_0, b_0,\lambda)$. Since $\underset{\gam\to 1+}{\lim} e^{-\frac{\gam+1}{2(\gam-1)}}=0$, this conditional monotonicity property holds for a bigger class of annuli as $\gam$ is close to $1$.
\end{remark}

\section{Proof of Theorem \ref{thmshock}}
This section is devoted to prove our main theorem.

\begin{proof}[Proof of Theorem $\ref{thmshock}$ (a)] (Step 1) Fix $\gam>1$, $r_0>0$, then fix  $\tilde r_1\in (0,r_0)$ to satisfy $\ln\frac{r_0}{\tilde r_1}<\frac{\gam+1}{2(\gam-1)}$. Given positive constant $(m_0,\kappa_0,M_0)$ with $M_0>1$, let $\ul E$ be from Proposition $\ref{proposition-supersonic}$ for $r_1$ replaced by $\tilde r_1$. For $E_0>\ul E$, let $(\rho_-, u_-, p_-, E_-)(t)$ be the solution to the initial value problem \eqref{system-simplified1} with $(M_-, E_-)(0)=(M_0, E_0)$. By Proposition \ref{proposition-supersonic}, \eqref{system-simplified1} with $(M_-, E_-)(0)=(M_0, E_0)$ has a unique solution in $[0,r_0-\tilde r_1]$, and the solution $(\rho_-, u_-, p_-, E_-)(t)$ satisfies
	\begin{equation*}
	M_-(t)>1,\quad M_-'(t)>0\quad\tx{for all $t\in[0, r_0-\tilde r_1]$.}
	\end{equation*}
Fix $t_s\in [0, T_1]$ for $T_1=r_0-\tilde r_1$, and let $(\rho_s, u_s, p_s, E_s)$ be given by \eqref{rhc}. Set $\kappa_s$ as in \eqref{definition-kappa},
and consider the following initial value problem
\begin{equation}
\label{ivp-ds-sub}
\begin{split}
&\begin{cases}
(\hat t\rho u, e^{S})=(m_0, \kappa_s)\\
\rho'=g_1(t, \rho, E, \kappa_s)\\
(\hat t E)'=g_2(t, \rho, E)
\end{cases},\\
&(\rho, E)(t_s)=(\rho_s, E_s)
\end{split}
\end{equation}
for $(g_1, g_2)$ defined by \eqref{def-gs}. Set $M_s:=u_s/\sqrt{\frac{\gam p_s}{\rho_s}}=u_s/\sqrt{\gam \kappa_s\rho_s^{\gam-1}}$, then \eqref{inequality-rhadm} implies that $g_1$ and $g_2$ are $C^1$ with respect to $(t,\rho, E)$ near $(t_s,\rho_s, E_s)$. Therefore, \eqref{ivp-ds-sub} has a unique $C^1$ solution $(\rho_+,u_+,p_+,E_+)(t;t_s)$  with $\rho_+>0$ for $t\in(t_s-\eps_0, t_s+\eps_0)$ for some small constant $\eps_0>0$. Furthermore, the solution is subsonic in the sense that $\gam\kappa_s\rho_+^{\gam-1}-\frac{m_0^2}{\hat t^2\rho_+^2}>0$, or equivalently $M_+^2-1<0$ on $(t_s-\eps_0, t_s+\eps_0)$ for $M_+^2=\frac{u_+^2}{\gam\kappa_s\rho_+^{\gam-1}}$.

Moreover, the solution $(\rho_+, u_+, p_+, E_+)(t;t_s)$ satisfies
\begin{equation}
\label{r-mt-at-shock}
\rho_+'(t_s;t_s)>0.
\end{equation}
This can be checked as follows: By \eqref{rhc}, \eqref{system-downstream}, \eqref{inequality-rhadm}, \eqref{equation-M-another} and Proposition \ref{proposition-supersonic}, we have
\begin{equation*}
\frac{dM_+^2}{dt}(t_s;t_s)=(\gam+1)\mu_0 (r_0-t_s)^{\frac{\gam-3}{\gam+1}}\frac{(M_s^2)^{\frac{2\gam}{\gam+1}}}{M_s^2-1}\left((r_0-t_s) E_-(t_s)-K_-(t_s)\right)<0.
\end{equation*}
Then, \eqref{r-mt-at-shock}  follows from \eqref{rM-expression} combined with the inequality right above. By \eqref{r-mt-at-shock}, there exists a small constant $\eps_1\in(0, \eps_0)$ depending on $(\gam,m_0,\kappa_0,M_0,E_0,r_0,t_s,b_0)$ such that $(M_+^2)'(t;t_s)<0$ thus $\rho_+'(t;t_s)>0$ hold by \eqref{rM-expression} on $[t_s-\eps_1, t_s+\eps_1]$.
\smallskip

(Step 2) For each $t_s\in [0,r_0-\tilde r_1]$,
let $l(t_s)>0$ be given to satisfy the following properties:
\begin{itemize}
\item[(i)] The initial value problem \eqref{ivp-ds-sub} is uniquely solvable for $t_s\le t<l(t_s)$;

\item[(ii)]
\begin{equation}
\label{prop-sub-rad}
\rho_+(t;t_s)>0,\quad \rho_+'(t;t_s)>0,\quad 0<M_+(t;t_s)<1
\end{equation}
for $t_s\le t<l(t_s)$;

\item[(iii)] (ii) does not hold for $t\ge l(t_s)$.
\end{itemize}
Note that we may have $l(t_s)=\infty$. Choose $T$ as
\begin{equation*}
T=\min\{\inf_{t_s\in[0, r_0-\tilde{r}_1]} l(t_s), r_0-\tilde{r}_1\}.
\end{equation*}
Such a constant $T$ is strictly positive.
Then, we obtain from Proposition \ref{promp} that $\frac{dp_+}{dt_s}(T;t_s)<0$ for all $t_s\in (0,r_0-r_1)$. This proves Theorem \ref{thmshock} (a).
\end{proof}

\begin{remark}
If $1<\gam<2$, a lengthy computation shows the initial value problem $\eqref{ivp-ds-sub}$ is uniquely solvable only on a finite interval especially when $E_s>0$ is sufficiently large, which can happen if $E_0>0$ is sufficiently large. Furthermore, life-span of $C^1$ solution $(\rho_+, u_+, p_+, E_+)(t;t_s)$ to \eqref{ivp-ds-sub} for $t\ge t_s$ converges to $0$ as $E_s\to \infty$. For $\gam\ge 2$, however, one can construct a family of radial transonic shock solutions satisfying \eqref{p-monot} in an annular domain $\nn$ given by \eqref{def-annulus} whenever \eqref{condition-ann-ration} holds.
\end{remark}

\begin{proof}
[Proof of Theorem \ref{thmshock} (b)]
(Step 1) Fix $\gam\ge 2$, and fix two constants $r_0>r_1>0$ with $\ln \frac{r_0}{r_1}<\frac{2(\gam-1)}{\gam+1}$ to define an annulus $\nn$ by \eqref{def-annulus}. Given positive constant $(m_0, \kappa_0, M_0)$ with $M_0>1$, let $\underline{E}$ be from Proposition \ref{proposition-supersonic}. For $E_0>\underline{E}$ to be further specified later, let $(\rho_-, u_-, p_-, E_-)(t)$ be the solution to the initial value problem \eqref{system-simplified1} with $(M_-, E_-)(0)=(M_0, E_0)$. By Proposition \ref{proposition-supersonic}, the solution $(\rho_-, u_-, p_-, E_-)(t)$ satisfies
\begin{equation*}
M_-(t)>1,\quad M_-'(t)>0\quad\tx{for all $t\in[0, T]$.}
\end{equation*}
Differently from the proof of Theorem \ref{thmshock}(a), a lower bound of $E_0$ needs to be adjusted further to acquire a family of radial transonic shock solutions satisfying \eqref{p-monot} for $t_s\in(0, T_*)$ for $T_*=r_0-r_1$.
\smallskip

(Step 2)
Fix $t_s\in(0, T_*)$, and let $(\rho_+,u_+,p_+,E_+)(t;t_s)$ be a $C^1$ solution to \eqref{ivp-ds-sub}.
As discussed in the proof of Theorem \ref{thmshock}(a), $(\rho_+,u_+,p_+,E_+)(t;t_s)$ is well defined on $(t_s-\eps_0, t_s+\eps_0)$ with satisfying \eqref{prop-sub-rad} for some small $\eps_0>0$.
\smallskip
In this step, we find a sufficient condition for $(\rho_+,u_+,p_+,E_+)(t;t_s)$ to satisfy
\begin{equation}\label{sub-condition}
\rho_+(t;t_s)>0,\quad \rho_+'(t;t_s)>0
,\quad\tx{and}\quad 0<M_+(t;t_s)<1 \quad\tx{for $t_s\le t\le T_*$}
\end{equation}
for any $t_s\in(0, T_*)$.

Rewrite $\rho^\p=g_1(t,\rho,E,\kappa_s)$ in $\eqref{def-gs}$ as
	\begin{align}\label{oderhorho}
	\rho^\p=\frac{\rho^{2-\gam}(\hat t E-\frac{m_0^2}{\hat t^2\rho^2})}{\hat t \gam \kappa_s (1-M^2)}.
	\end{align}
	One can see from $\eqref{oderhorho}$ that $\eqref{sub-condition}$ is equivalent to $\rho_+(t;t_s)>0$, $\hat tE_+(t;t_s)-\frac{m_0^2}{\hat t^2\rho_+^2(t;t_s)}>0$ and $ 0<M_+(t;t_s)<1$ for $t\in [t_s, r_0-r_1]$.

Integrate the differential inequality $(\hat t E)'\ge -\hat t b_0$ over the interval $[0, t]$ to get
\begin{equation}
\label{inequality-lb-Eds}
E(t)\ge \frac{1}{\hat t}\left((r_0-t_s)E_s-\beta_1\right)
\end{equation}
for some constant $\beta_1>0$ depending only on $(r_0, r_1, b_0)$.

Setting as $\hat{t}_s=r_0-t_s$, the estimate $\eqref{inequality-lb-Eds}$ yields
	\begin{align}\label{teu2-estimate-crutial}
	(\hat t E_+-\frac{m_0^2}{\hat t^2\rho_+^2})(t;t_s)&> \hat t_sE_s-\beta_1-\frac{m_0^2}{\hat t^2\rho_s^2}\nonumber\\
	&\ge \hat t_s E_s-\frac{\hat t_s^2}{r_1^2}u_s^2-\beta_1=:F_s(t_s)\qd\tx{for $t\in [t_s,t_s+\eps_0)$}.
	\end{align}

If
\begin{equation}\label{nece-condi1}
  F_s(t_s)\ge \beta_1,
\end{equation}
then we have
\begin{equation}
\label{Msigma}
  \limsup_{t\to (t_s+\eps_0)-}M_+^2(t;t_s)\le 1-\delta_s
\end{equation}
for some constant $\delta_s\in(0,1)$. This can be checked as follows: We rewrite $\eqref{odeM}$ as
	\begin{equation}\label{odeMreu}
\begin{split}
	(M^2_+)'&=\mu_s\hat t^{\frac{\gam-3}{\gam+1}}\frac{(M_+^2)^{\frac{2\gam}{\gam+1}}}{M_+^2-1}
\left({\hat t E_+-\frac{m_0^2}{\hat t^2\rho_+^2}}+\frac{2(M_+^2-1)}{\mu_s (\hat t^2 M_+^2)^{\frac{\gam-1}{\gam+1}}}\right)\\
&=:\mu_s\hat t^{\frac{\gam-3}{\gam+1}}\frac{(M_+^2)^{\frac{2\gam}{\gam+1}}}{M_+^2-1}H_1(M_+, \rho_+, E_+, \kappa_s)
\end{split}	
	\end{equation}
	for $\mu_s:=(\gam+1)\frac{1}{\gam\kappa_s}\left(\frac{\gam \kappa_s}{m_0^2}\right)^{\frac{\gam-1}{\gam+1}}$. If $F_s(t_s)\ge \beta_1$, then \eqref{teu2-estimate-crutial} implies that
\begin{equation*}
  \lim_{M_+\rightarrow 1-} H_1(M_+, \rho_+, E_+, \kappa_s)\ge \beta_1>0,
\end{equation*}
thus $M_+^2$ decreases when $M_+^2$ is less than $1$, and sufficiently close to 1. This proves \eqref{Msigma}.
\smallskip

Next, we show that
\begin{equation}
\label{no-fin-bl}
 \rho_s< \lim_{t\to (t_s+\eps_0)-}\rho_+(t;t_s)<+\infty.
\end{equation}
Since $\rho_+(t;t_s)$ monotonically increases over the interval $(t_s, t_s+\eps_0)$, by using \eqref{def-gs}, \eqref{system-downstream}, \eqref{oderhorho} and \eqref{Msigma}, we get
\begin{equation}
\label{diff-inequ-sub}
 0< \rho_+'(t;t_s)\le \frac{C_s(1+E_s)}{\kappa_s\delta_s}\rho_+^{3-\gam}(t;t_s)\quad\tx{for $t\in(t_s, t_s+\eps_0)$}
\end{equation} 	
for some constant $C_s>0$ depending on $(\gam, r_0,r_1, t_s, \rho_s)$. Here, we used the fact that $\hat t E_+(t;t_s)\le \hat t_s E_s+\rho_+(t;t_s)\int_{t_s}^t (r_0-\eta)\,d\eta$ provided that $\rho_+$ monotonically increases up to $t$. This is obtained from the equation $(\hat t E_+)'=\hat t(\rho_+-b(t))<\hat t \rho_+$. Since $\gam\ge 2$, \eqref{no-fin-bl} is obtained from \eqref{diff-inequ-sub} and the monotone sequence theorem. In \eqref{teu2-estimate-crutial} and \eqref{Msigma}, the estimates are independent of $\eps_0$. And, \eqref{diff-inequ-sub} shows that $\rho_+(t;t_s)$ does not blow up for $t\le T_*(<r_0)$ when $\rho_+$ monotonically increases. Therefore, by the method of continuation, the following lemma is obtained:
\begin{lem}
\label{lemma-nece-cond}
  For a fixed $t_s\in(0, T_*)$, if the condition \eqref{nece-condi1} holds, then the corresponding subsonic solution $(\rho_+, u_+, p_+, E_+)(t;t_s)$ is well defined on the interval $[t_s, T_*]$, and it satisfies the properties \eqref{sub-condition}.
\end{lem}

	(Step 3) By Proposition \ref{promp} and Lemma \ref{lemma-nece-cond}, if we find $E_*\in[\ul E, \infty)$ depending only on $(\gam, r_0, r_1, \rho_0, u_0, p_0, b_0)$ to satisfy $F_s(t_s)\ge \beta_1$, then Theorem \ref{thmshock}(b) is proved. The rest of the proof is devoted to find such a constant $E_*$.

	\begin{lem}
\label{lemma-claim2}
	For each $t_s\in (0, T_*)$, set $q_s:=\frac{r_0}{\hat t_s}$.	
	For any given $\lambda>0$ and $\sigma>0$,
	there exists a constant $\ul E_*\in[\underline E, \infty)$ depending on
	$(\gam, r_0, r_1, \rho_0, u_0, p_0, b_0)$ and on $(t_s, \lambda, \sigma)$ such that, whenever $E_0\ge \ul E_*$, the solution to $\eqref{system-simplified1}$ with $(M_-,E_-)(0)=(M_0,E_0)$ satisfies
	\begin{align}\label{claim2}
	\hat t E_-(t)-\left(\frac{1}{2\ln(q_s+\sigma)}\right)\frac{\gam+1}{\gam-1}K_-(t)>\lambda\qd\tx{for all $t\in [0,t_s]$}.
	\end{align}

	\begin{proof}
	By Proposition \ref{proposition-supersonic}, for $E_0>\ul E$, the solution to $\eqref{system-simplified1}$ uniquely exists and satisfies $M_->1$ and $M_-^\p>0$ in $[0,T_*]$. Then, by \eqref{new-equation-E} and  $\eqref{kprime}$, we obtain that
	\begin{align*}
	\left(\hat t E_--\left(\frac{1}{2\ln(q_s+\sigma)}\right)\frac{\gam+1}{\gam-1}K_-
\right)^\p =\hat t\left(\mu_1\left(\frac{1}{\hat t^2M_-^2}\right)^{\frac{1}{\gam+1}}-b(t)\right)-\frac{E_-}{\ln(q_s+\sigma)}.
	\end{align*}
	Combine this with \eqref{E-Sup-ineq-2} to have
	\begin{align}
	\label{thm-pf-inequa1}
	\left(\hat t E_--\left(\frac{1}{2\ln(q_s+\sigma)}\right)
\frac{\gam+1}{\gam-1}K_-\right)(t)>r_0E_0\left(1-\frac{\ln q_s}{\ln(q_s+\sigma)}\right)-\alpha_2
	\end{align}
	for all $t\in [0,t_s]$, where a constant $\alpha_2>0$ is chosen depending on $(\gam, r_0, r_1, \rho_0, u_0, p_0, b_0)$.
	Since $1-\frac{\ln q_s}{\ln(q_s+\sigma)}>0$, one can choose a constant $\ul E_*\in[\ul E, \infty)$ depending on $(\gam, r_0, r_1, \rho_0, u_0, p_0, b_0)$ and on $(t_s, \lambda, \sigma)$ such that \eqref{thm-pf-inequa1} implies \eqref{claim2} whenever
	$E_0 \ge \ul E_*$.
	\end{proof}
	\end{lem}
	
 	(Step 4) 
	A direct computation with using $\eqref{rhc}$ yields that
 	\begin{align*}
 	u_s=u_-(t_s)\left(\frac{\gam-1}{\gam+1}+\frac{2}{(\gam+1)M_-^2(t_s)}\right).
 	\end{align*}
	Combine this expression with $u_-^2<\frac{\gam+1}{\gam-1}K_-$, $E_s=E_-(t_s)$, and $\ln \frac{r_0}{r_1}<\frac{2(\gam-1)}{\gam+1}$ to get
	
\begin{equation}
\label{tses1}
\begin{split}
F_s(t_s)
&>\hat t_s E_-(t_s)-\frac{\hat t_s^2}{r_1^2}
\left(\frac{\gam-1}{\gam+1}+\frac{2}{(\gam+1)M_-^2(t_s)}\right)^2\frac{\gam+1}{\gam-1}K_-(t_s)-\beta_1\\
	&>	\hat t_s E_-(t_s)-\left(\frac{e^{\frac{\gam+1}{2(\gam-1)}}}{q_s}\right)^2\left(\frac{\gam-1}{\gam+1}+\frac{2}{(\gam+1)M_-^2(t_s)}\right)^2\frac{\gam+1}{\gam-1}K_-(t_s)-\beta_1=:G_s(t_s)
	\end{split}
	\end{equation}
	for $F_s(t_s)$ from  \eqref{oderhorho}.

By $\eqref{teu2-estimate-crutial}$ and $\eqref{tses1}$, if $G_s(t_s)>0$ for all $t_s\in (0, T_*)$, then $(\rho_+, u_+, p_+, E_+)(t;t_s)$ satisfies \eqref{prop-sub-rad} for $t_s\le t<T_*$, for each $t_s\in(0, T_*)$ so that Theorem \ref{thmshock}(b) is proved.


It remains to find $\ul E_{\flat}\in [\ul E, \infty)$ depending on $(\gam, r_0, r_1, \rho_0, u_0, p_0, b_0)$ so that $G_s(t_s)>0$ for all $t_s\in(0, t_*)$, whenever $E_0\ge \ul E_{\flat}$. To find such a constant $\ul E_{\flat}$, the following lemma is needed.
\begin{lem}
\label{lemma-pineq}
For $\gam \ge 2$, it holds that
\begin{align}
	\label{pineq}
	2\Bigl(\frac{\gam-1}{\gam+1}\Bigr)^2
	 \Bigl(\frac{e^{\frac {\gam+1}{2(\gam-1)}}}{\xi}\Bigr)^2\ln \xi<1
		\end{align}
	for any $\xi\in[1, e^{\frac{\gam+1}{2(\gam-1)}})$.
	\begin{proof}
	
	Define $f(\xi):=2(\frac{\gam-1}{\gam+1})^2 {e^{\frac {\gam+1}{(\gam-1)}}} \frac{\ln \xi}{\xi^2} $ for $\xi>0$.
	Since
	$f'(\xi)=2e^{\frac{\gam-1}{\gam+1}}(\frac{\gam-1}{\gam+1})^2\frac{1}{\xi^3}(1-2\ln \xi),$ the maximum value of $f$ is acquired at $\xi=e^{\frac 12}$, that is, $f(\xi)\le f(e^{\frac 1 2})=e^{\frac{2}{\gam-1}}(\frac{\gam-1}{\gam+1})^2$ for $\xi>0$. Therefore, Lemma \ref{lemma-pineq} is proved if we show that $e^{\frac{2}{\gam-1}}(\frac{\gam-1}{\gam+1})^2<1$ for all $\gam\ge 2$.

	Set $g(\gam):=e^{\frac{1}{\gam-1}}(\frac{\gam-1}{\gam+1})$, then we have
	\begin{equation*}
	 g(2)=\frac{e}{3} <1\,\quad \lim_{\gam\to \infty} g(\gam)=1,
	 \quad\tx{and} \,\, \rm{sgn}\, g'(\gam)=\rm{sgn}\,(\gam-3).
	 \end{equation*}
	 This implies that $f(e^{\frac 1 2})=g(\gam)^2<1$ for $\gam\ge 2 $.

	\end{proof}

\end{lem}

(Step 5) Suppose that $t_s\in[\zeta, T_*)$ for some $\zeta\in(0, T_*)$.
\smallskip

{\emph{Claim. There exists a constant $\underline{E}^{*}_{(\zeta)}$ depending on $(\gam, r_0, r_1, \rho_0, u_0, p_0, b_0, \zeta)$ such that whenever $E_0\ge \underline{E}^{*}_{(\zeta)}$, we have
\begin{equation}
\label{gs-zeta}
G_s(t_s)>\beta_1\quad\tx{for $t_s\in[\zeta, T_*)$}
\end{equation}
for the constant $\beta_1>0$ from \eqref{inequality-lb-Eds}.
}}
	\smallskip

By Lemma \ref{lemma-pineq}, there exist constants $\sigma_{\gam}>0$ and $\underline M_{\gam}>1$ depending only on $\gam$ such that
	\begin{align}\label{qineq}
	\left(\frac{e^{\frac{\gam+1}{2(\gam-1)}}}{q_s}\right)^2
\left(\frac{\gam-1}{\gam+1}+\frac{2}{(\gam+1)\underbar M_{\gam}^2}\right)^2<\frac{1}{2\ln (q_s+\sigma_{\gam})} \qd\tx{for all $t_s\in [\zeta, T_*)$}.
	\end{align}

One can choose $E^{(1)}_{\zeta} \in[\ul E, \infty)$ depending on $(\gam, r_0, r_1, \rho_0, u_0, p_0, b_0, \zeta)$ so that 	
	whenever $E_0> E^{(1)}_{\zeta}$, the solution to $\eqref{system-simplified1}$ with $(M_-,E_-)(0)=(M_0,E_0)$ satisfies
			\begin{align}
\label{inequaliy-M-zeta}
			M_-^2(t_s)\ge M_-^2(\zeta)>\underline M_{\gam}^2\quad\tx{for all $t_s\in[\zeta, T_*).$}
			\end{align}
This can be verified as follows: Following the proof of Proposition $\ref{proposition-supersonic}$, one can directly check that, for any constant $\lambda_1>0$, there exists a constant $\ul{E}_{\lambda_1}\in [\ul E, \infty)$ depending on $(\gam, r_0, r_1, \rho_0, u_0, p_0, b_0, \lambda_1)$ such that, whenever $E_0\ge \ul E_{\lambda_1}$, the solution to $\eqref{system-simplified1}$ with $(M_-,E_-)(0)=(M_0,E_0)$ satisfies $\hat tE_- -K_->\lambda_1$ for all $t\in [0,T_*]$. Since, for $E_0>\underline E$, $M_-^\p>0$ in $[0,r_0-r_1]$, by $\eqref{equation-M-another}$,
	\begin{align*}
	(M^2)^\p> \mu_2 (M^2)^{\frac{\gam-1}{\gam+1}} \lambda_1\qd\tx{in $[0,T]$}
	\end{align*}
	for a constant $\mu_2>0$ depending only on $(\gam, r_0, r_1, \rho_0, u_0, p_0)$. By integrating this inequality over the interval $[0,\zeta]$, it is obtained that
\begin{equation*}	(M_-^2)^{\frac{2}{\gam+1}}(\zeta)
>(M_0^2)^{\frac{2}{\gam+1}}+\frac{2\lambda_1\mu_2\zeta}{\gam+1},
\end{equation*}
therefore, one can choose $\lambda_1>0$ large depending on $(\gam, r_0, r_1, \rho_0, u_0, p_0, \zeta)$ to satisfy \eqref{inequaliy-M-zeta}. For such $\lambda_1$, we choose as $E^{(1)}_{\zeta}=\ul{E}_{\lambda_1}$.

	By \eqref{qineq} and \eqref{inequaliy-M-zeta}, if $E_0\ge E^{(1)}_{\zeta}$, then we get
	\begin{align}\label{rek5}\begin{split}
	G_s(t_s)>t_sE_-(t_s)-\left(\frac{1}{2\ln (q_s+\sigma_\gam)}\right)\frac{\gam+1}{\gam-1}K_-(t_s)-\beta_1\qd\tx{for all $t_s\in [\zeta,T]$}.
	\end{split}
	\end{align}
We apply Lemma \ref{lemma-claim2} with choosing $\lambda=2\beta_1$ to conclude that there exists a constant $E^{(2)}_{\zeta}\in[E^{(1)}_{\zeta}, \infty)$ depending on $(\gam, r_0, r_1, \rho_0, u_0, p_0, b_0, \zeta)$ such that
\begin{equation}
\label{rek6}
  G_s(t_s)>\beta_1\quad\tx{for all $t_s\in[\zeta, T_*)$.}
\end{equation}
By choosing as $\ul E^*_{(\zeta)}= E^{(2)}_{\zeta}$, {\emph{Claim}} is verified.
	

(Step 6) {\emph{Claim: One can choose constants $\zeta_*\in(0, T_*)$ and $E^{(3)}_{\zeta_*}\in[\ul E, \infty)$ depending on $(\gam, r_0, r_1, \rho_0, u_0, p_0, b_0)$ so that whenever $E_0\ge E^{(3)}_{\zeta_*}$, we have
\begin{equation*}
  G_s(t_s)>\beta_1\quad\tx{for all $t_s\in (0, \zeta_*]$.}
\end{equation*}}}

For $q_s=\frac{r_0}{r_0-t_s}\in (1,\frac{r_0}{r_0-\zeta_*}]$, or equivalently $t_s\in(0,\zeta_*]$, we have
	\begin{equation}\label{rek2}
	G_s(t_s)>
\hat t_s E_-(t_s)-e^{\frac{\gam+1}{\gam-1}}\left(\frac{\gam-1}{\gam+1}
+\frac{2}{(\gam+1)M_-^2(t_s)}\right)^2\frac{\gam+1}{\gam-1}K_-(t_s)-\beta_1.
	\end{equation}
	Choose $\zeta_*\in(0, T_*)$ sufficiently small to satisfy	$\frac{1}{2\ln\left(\frac{r_0}{r_0-\zeta_*}\right)}
>e^{\frac{\gam+1}{\gam-1}}
\left(\frac{\gam-1}{\gam+1}+\frac{2}{(\gam+1)M_0^2}\right)$, then choose a constant $\sigma_{\zeta_*}>0$ to satisfy	$\frac{1}{2\ln\left(\frac{r_0}{r_0-\zeta_*}+\sigma_{\zeta_*}\right)}>e^{\frac{\gam+1}{\gam-1}}\left(\frac{\gam-1}{\gam+1}+\frac{2}{(\gam+1)M_0^2}\right)$.
Such a constant $\zeta_*$ can be chosen depending on $(\gam, r_0, \rho_0, u_0, p_0)$, and so does $\sigma_{\zeta_*}$.	
By \eqref{rek2}, the choice of $\zeta_*$, and Lemma \ref{lemma-claim2}, there exists a constant $E^{(3)}_{\zeta_*}\in[\ul E, \infty)$ depending only on $(\gam, r_0, r_1, \rho_0, u_0, p_0, b_0)$ so that whenever $E_0\ge E^{(3)}_{\zeta_*}$, we have
	\begin{align}\label{rek4}
	\begin{split}
	G_s(t_s)&>\hat t_s E_-(t_s)-e^{\frac{\gam+1}{\gam-1}}\left(\frac{\gam-1}{\gam+1}+\frac{2}{(\gam+1)M_0^2}\right)^2\frac{\gam+1}{\gam-1}K_-(t_s)-\beta_1
	\\&>\hat t E_-(t_s)-\left(\frac{1}{2\ln(\frac{r_0}{r_0-\zeta_*}+\sigma_{\zeta_*})}\right)\frac{\gam+1}{\gam-1}K_-(t_s) -\beta_1\\
&>\beta_1 \qd\tx{for all $t_s\in [0,\zeta_*]$}.
	\end{split}
	\end{align}
The claim is verified.

(Step 7) Finally, we choose $E_*$ for Theorem \ref{thmshock}(b) as
\begin{equation*}
  E_*=\max\{\ul E^*_{(\zeta_*/2)}, E^{(3)}_{\zeta_*}\},
\end{equation*}
for $\ul E^*_{(\zeta)}$ from step 5, so that the function $G_s(t_s)$ from \eqref{tses1} satisfies $G_s(t_s)\ge \beta_1$ for all $t_s\in (0, T_*)$, thus the condition \eqref{nece-condi1} holds for all $t_s\in (0, T_*)$ whenever $E_0\ge E_*$.

	This finishes the proof of Theorem \ref{thmshock} (b).
\end{proof}


\end{document}